\documentclass[a4paper,12pt]{amsart}
\usepackage{amsfonts}
\usepackage{amsmath}
\usepackage{amssymb}
\usepackage[a4paper]{geometry}
\usepackage{mathrsfs}
\usepackage{hyperref}
\renewcommand\eqref[1]{(\ref{#1})} %Need with hyperref
\numberwithin{equation}{section}
\theoremstyle{plain}
\newtheorem{thm}{Theorem}[section]

\theoremstyle{definition}

\newtheorem{rem}[thm]{Remark}

\begin{document}

   \title[Boundary conditions for the Kohn Laplacian]
   %{Integral boundary conditions for the sub-Laplacian on the Heisenberg group}
   {On Kac's principle of not feeling the boundary for the Kohn Laplacian on the Heisenberg group}
 % {On Kac's principle of not feeling the boundary for the Kohn Laplacian and its powers}
%\title{M. Kac's principle of not feeling the boundary for the sub-Laplacian on the Heisenberg group}

\author[Michael Ruzhansky]{Michael Ruzhansky}
\address{
  Michael Ruzhansky:
  \endgraf
  Department of Mathematics
  \endgraf
  Imperial College London
  \endgraf
  180 Queen's Gate, London SW7 2AZ
  \endgraf
  United Kingdom
  \endgraf
  {\it E-mail address} {\rm m.ruzhansky@imperial.ac.uk}
  }
\author[Durvudkhan Suragan]{Durvudkhan Suragan}
\address{
  Durvudkhan Suragan:
  \endgraf
  Institute of Mathematics and Mathematical Modelling
  \endgraf
  125 Pushkin str.
  \endgraf
  050010 Almaty
  \endgraf
  Kazakhstan
  \endgraf
  {\it E-mail address} {\rm suragan@math.kz}
  }

\thanks{The authors were supported in parts by the EPSRC
 grant EP/K039407/1 and by the Leverhulme Grant RPG-2014-02,
 as well as by the MESRK grant 5127/GF4.}

     \keywords{sub-Laplacian, Kohn Laplacian, integral boundary conditions, Heisenberg group, Newton potential}
     \subjclass{35R03 35S15}

     \begin{abstract}
     In this note we construct an integral boundary condition for the Kohn Laplacian in a given domain on the Heisenberg
     group extending to the setting of the Heisenberg group M.~Kac's ``principle of not feeling the boundary''.
     This also amounts to finding the trace on smooth
     surfaces of the Newton potential associated to the Kohn Laplacian.
     We also obtain similar results for higher powers of the Kohn Laplacian.
     \end{abstract}
     \maketitle

\section{Introduction}

In a bounded domain of the Euclidean space $\Omega\subset\mathbb R^{d},\,\, d\geq 2,$ it is very well known that the
solution to the Laplacian equation
\begin{equation}
\Delta u(x)=f(x), \,\,\,\ x\in\Omega,\label{15}
\end{equation}
is given by the Green formula (or the Newton potential formula)
\begin{equation}
u(x)=\int_{\Omega}\varepsilon_{d}(x-y)f(y)dy,\,\, x\in\Omega,
\label{14}
\end{equation}
for suitable functions $f$ supported in $\Omega$.
Here $\varepsilon_{d}$ is the fundamental solution to $\Delta$ in $\mathbb R^d$ given by
\begin{equation}
\varepsilon_{d}(x-y)=\left\{
\begin{array}{ll}
    \frac{1}{(2-d)s_{d}}\frac{1}{|x-y|^{d-2}},\,\,d\geq 3,\\
    \frac{1}{2\pi}\log|x-y|, \,\,d=2, \\
\end{array}
\right.
\label{EQ:fs}
\end{equation}
where $s_{d}=\frac{2\pi^{\frac{d}{2}}}{\Gamma(\frac{d}{2})}$ is the surface area of the unit sphere in
$\mathbb R^{d}$.

An interesting  question having several important applications is what boundary conditions can be put on $u$ on
the (smooth) boundary $\partial\Omega$ so that equation \eqref{15} complemented by this
boundary condition would have the solution in $\Omega$ still given by the same formula \eqref{14},
with the same kernel $\varepsilon_d$ given by \eqref{EQ:fs}.
It turns out that the answer to this question is the integral boundary condition
\begin{equation}
-\frac{1}{2}u(x)+\int_{\partial\Omega}\frac{\partial\varepsilon_{d}(x-y)}{\partial n_{y}}u(y)d S_{y}-
\int_{\partial\Omega}\varepsilon_{d}(x-y)\frac{\partial u(y)}{\partial n_{y}}d S_{y}=0,\,\,
x\in\partial\Omega,
\label{16}
\end{equation}
where $\frac{\partial}{\partial n_{y}}$ denotes the outer normal
derivative at a point $y$ on $\partial\Omega$.
A converse question to the one above would be to determine the trace of the Newton potential
\eqref{14} on the boundary surface $\partial\Omega$, and one can use the potential theory to show that it
has to be given by \eqref{16}.

The boundary condition \eqref{16} appeared in M. Kac's work \cite{Kac1} where
he called it and the subsequent spectral analysis
``the principle of not feeling the boundary''. This was further expanded in Kac's book \cite{Kac2}
with several further applications to the spectral theory and the asymptotics of the Weyl's eigenvalue
counting function.
In \cite{KS1} by using the boundary condition \eqref{16} the eigenvalues and eigenfunctions of the Newton potential
\eqref{14} were explicitly calculated in the 2-disk and in the 3-ball.
In general, the boundary value problem \eqref{15}-\eqref{16} has various interesting properties and applications
(see, for example, Kac \cite{Kac1,Kac2} and Saito \cite{Sa}).
The boundary value problem \eqref{15}-\eqref{16} can also be generalised for higher degrees of the
Laplacian, see \cite{KS2, KS3}.

\medskip
In this note we are interested in and we give answers to the following questions:

\begin{itemize}
\item What happens if an elliptic operator (the Laplacian) in \eqref{15} is replaced by a hypoelliptic operator?
We will realise this as a model of replacing the Euclidean space by the Heisenberg group and the Laplacian on
$\mathbb R^d$ by a sub-Laplacian (or the Kohn-Laplacian) on $\mathbb H_{n-1}$. We will show that the
boundary condition \eqref{16} is replaced by the integral boundary condition \eqref{7} in this setting
(see also \eqref{EQ:bc0}).

\item Since the theory of boundary value problems for elliptic operators is well understood, we know that the single
condition \eqref{16} on the boundary $\partial\Omega$ of a bounded domain $\Omega$ guarantees the unique
solvability of the equation \eqref{15} in $\Omega$. Is this uniqueness preserved in the hypoelliptic model as well
for a suitably chosen replacement of the boundary condition \eqref{16}? The case of the second order operators
is favourable from this point of view due to the validity of the maximum principle, see Bony \cite{bony_69}.
The Dirichlet problem has been considered by Jerison \cite{J}.
The answer in the case of the boundary value problem in our setting is given in Theorem \ref{THM:main}.

\item What happens if we consider the above questions for higher order equations? In general, it is known
that for higher order Rockland operators on stratified groups, fundamental solutions may be not unique,
see Folland \cite{folland_75} and Geller \cite{geller_83}, and for a unifying discussion see also the book
\cite{FR}. However, for powers of the Kohn Laplacian we still have the uniqueness provided that we impose
higher order boundary conditions in a suitable way, see Theorem \ref{THM:main2}.
\end{itemize}
We now describe the setting of this paper.
The Heisenberg group $\mathbb{H}_{n-1}$ is the space $\mathbb C^{n-1}\times \mathbb R$ with the group operation
given by
 \begin{equation}
 (\zeta,t)\circ (\eta,\tau)=(\zeta+\eta, t+\tau+ 2\,{\rm Im}\,\zeta\eta) \label{1},
 \end{equation}
for $ (\zeta,t), (\eta,\tau) \in \mathbb C^{n-1}\times \mathbb R$.
Writing $\zeta=x+iy$ with $x_{j}, y_{j}, j=1,...,n-1,$ the real coordinates on $\mathbb{H}_{n-1}$,
the left-invariant vector fields
$$\tilde X_{j}=\frac{\partial}{\partial x_{j}}+2y_{j}\frac{\partial}{\partial t}, \quad j=1,..., n-1,$$
$$\tilde Y_{j}=\frac{\partial}{\partial y_{j}}-2x_{j}\frac{\partial}{\partial t}, \quad j=1,..., n-1,$$
$$T=\frac{\partial}{\partial t},$$
form a basis for the Lie algebra $\mathfrak{h}_{n-1}$ of $\mathbb{H}_{n-1}$.

On the other hand, $\mathbb{H}_{n-1}$ can be viewed as the boundary of the Siegel upper half space
in $\mathbb C^{n}$,
$$
\mathbb{H}_{n-1}=\{(\zeta,z_{n})\in \mathbb C^{n}: {\rm Im}\, z_{n}=|\zeta|^{2}, \zeta=(z_{1},..., z_{n-1})\}.
$$
Parameterizing $\mathbb{H}_{n-1}$ by $z=(\zeta,t)$ where $t={\rm Re}\, z_{n}$, a basis for the complex tangent space of
$\mathbb{H}_{n-1}$ at the point $z$ is given by the left-invariant vector fields
$$
X_{j}=\frac{\partial}{\partial z_{j}}+i\overline{z}\frac{\partial}{\partial t},\quad j=1,\ldots,n-1.
\label{2}
$$
We denote their conjugates by
$X_{\overline{j}}\equiv \overline{X}_j=\frac{\partial}{\partial \overline{z}_{j}}-iz\frac{\partial}{\partial t}.$
The operator
\begin{equation} \Box_{a,b}=\sum_{j=1}^{n-1} (aX_{j}X_{\overline{j}}+bX_{\overline{j}}X_{j}), \quad a+b=n-1,\end{equation}
is a left-invariant, rotation invariant differential operator that is homogeneous
of degree two
(cf. \cite{FK}).
This operator is a slight generalisation of the standard
sub-Laplacian or Kohn-Laplacian $\Box_b$ on the Heisenberg group $\mathbb{H}_{n-1}$ which,
when acting on the coefficients of a $(0,q)$-form can be written as
$$
\Box_b=-\frac{1}{n-1}\sum_{j=1}^{n-1} ((n-1-q) X_j X_{\overline{j}}+q X_{\overline{j}} X_j).
$$

Folland and Stein \cite{FS} found that a fundamental solution of the operator $\Box_{a,b}$ is
a constant multiple of
\begin{equation}
\varepsilon(z)= \varepsilon(\zeta,t)=\frac{1}{(t+i|\zeta|^{2})^{a}(t-i|\zeta|^{2})^{b}},
\label{3}
\end{equation}
and defined the Newton potential (volume potential) for a function $f$ with compact support
contained in a set $\Omega\subset \mathbb{H}_{n-1}$ by
\begin{equation}
u(z)=\int_{\Omega}f(\xi)\varepsilon(\xi^{-1}z)d\nu(\xi),
\label{4}
\end{equation}
with $d\nu$ being the volume element (the Haar measure on $\mathbb{H}_{n-1}$), coinciding with the Lebesgue measure on
$\mathbb C^{n-1}\times \mathbb R$.
More precisely, they proved that
$$
\Box_{a,b} u= c_{a,b}f,
$$
where the constant $c_{a,b}$ is zero if $a$ and $b= -1,-2,\ldots,n,n+1,\ldots,$
and $c_{a,b}\not=0$ if $a$ or $b\not= -1,-2,\ldots,n,n+1,\ldots$
In fact, then we can take
$$c_{a,b}=\frac{2(a^{2}+b^{2}){\rm Vol}(B_{1})}{(2i)^{n}}
   \frac{(n-1)!}{a(a-1)...(a-n
    +1)}(1-{\rm exp}(-2ia\pi))$$
 for $a\not\in\mathbb Z$, see the proof of Theorem 1.6 in Romero \cite{R}.
Similar conclusions by a different methods were obtained by Greiner and Stein \cite{GS}.
For a more general analysis of fundamental solutions for sub-Laplacians we can refer to
Folland \cite{folland_75} as well as to a discussion and references in
Stein \cite{St}. The Kohn Laplacian and its generalisations may be considered as natural
models for dealing with sums of squares also on more general manifolds, as it is now well known,
see e.g. Rothschild and Stein \cite{RS}.

In the above notation, the distribution $\frac{1}{c_{a,b}}\varepsilon$ is the fundamental solution of
$\Box_{a,b}$, while $\varepsilon$ satisfies the equation
\begin{equation}\label{EQ:vare}
\Box_{a,b}\varepsilon=c_{a,b}\delta.
\end{equation}
However, although we could have rescaled $\varepsilon$ for it to become the fundamental
solution, we prefer to keep the notation yielding \eqref{EQ:vare} in order to follow the
notation of \cite{FS} and \cite{R} to be able to refer to their results directly.

\medskip
Throughout this paper we assume that $c_{a,b}\neq 0$, i.e. both
$$a \textrm{ and } b\neq -1,-2,\ldots,n,n+1,\ldots.$$
In addition, without loss of generality we may also assume that $a,b\geq0$.

\medskip
Now, in analogy to the elliptic boundary value problem \eqref{15}--\eqref{16} for the Laplacian
$\Delta$ in $\mathbb R^d$, we consider the hypoelliptic boundary value problem for
the sub-Laplacian $\Box_{a,b}$ on $\mathbb{H}_{n-1}$, namely the equation
\begin{equation}
\Box_{a,b} u= c_{a,b}f
\label{EQ:BV}
\end{equation}
in a bounded set $\Omega\subset \mathbb{H}_{n-1}$ with smooth boundary $\partial\Omega$.
The first aim of this paper is to find a boundary condition of the Newton potential $u$ on
$\partial\Omega$ such that with this boundary condition the equation \eqref{EQ:BV} has a unique solution,
which is the Newton potential \eqref{4}.

Basing our arguments on the analysis of Folland and Stein \cite{FS} and Romero \cite{R}
we show that the
boundary condition \eqref{16} for the Laplacian in $\mathbb R^d$
is now replaced by the integral boundary condition \eqref{7} in this setting,
namely by the condition
\begin{equation}\label{EQ:bc0}
(c_{a,b}-H.R(z))u(z) -\int_{\partial \Omega} \varepsilon(\xi,z) \langle \nabla^{b,a}u(\xi), d\nu(\xi)\rangle +p.v. \, Wu(z)=0,
\; z\in\partial\Omega,
\end{equation}
on the boundary $\partial\Omega$, where
$H.R(z)$ is the so-called half residue, and where the second and the third term can be interpreted as
coming from the
suitably defined respectively single and double layer potentials $S$ and $W$ for the problem.
See Section \ref{SEC:subL} for the definitions and the precise statement.

\medskip
In Section \ref{SEC:subL} by using properties of fundamental solutions we construct a well-posed boundary value problem
for the differential equation \eqref{EQ:BV} with the required properties.
In Section \ref{SEC:higher} we generalise this result for higher powers of the Kohn Laplacian.
Throughout this paper we may use notations from \cite{R}, \cite{M} and \cite{RT}.

\section{The Kohn Laplacian}
\label{SEC:subL}

Let $\Omega\subset \mathbb{H}_{n-1}$ be an open bounded domain with a smooth boundary $\partial\Omega \in C^{\infty}$.
Consider the following analogy of the Newton potential on the Heisenberg group
\begin{equation}
u(z)=\int_{\Omega}f(\xi)
\varepsilon(\xi,z)d\nu(\xi)\quad \textrm{ in } \Omega,
\label{6}
\end{equation}
where $\varepsilon(\xi,z)=\varepsilon(\xi^{-1}z)$ is the rescaled
fundamental solution \eqref{3} of the sub-Laplacian, satisfying \eqref{EQ:vare}.
 As we mentioned $u$ is a solution of \eqref{EQ:BV} in $\Omega$. The aim of this section is to find a boundary condition
 for $u$ such that with this boundary condition the equation \eqref{EQ:BV} has a unique solution in $C^{2}(\Omega)$, say, and this solution is the Newton potential \eqref{6}.

\medskip
We recall a few notions and properties first. For $z=(\zeta,t)\in  \mathbb{H}_{n-1}$, we define its norm
by $|z|:=(|\zeta|^4+|t|^2)^{1/4}$. As any (quasi-)norm on $\mathbb{H}_{n-1}$, this satisfies a triangle
inequality with a constant, and allows for a polar decomposition. For $0<\alpha<1$, Folland and Stein
\cite{FS} defined the anisotropic H\"older spaces $\Gamma_\alpha(\Omega)$
by
$$
\Gamma_\alpha(\Omega)=\left\{
f:\Omega\to\mathbb C:\; \sup_{\stackrel{z_1,z_2\in \Omega}{z_1\not= z_2}}
\frac{|f(z_2)-f(z_1)|}{|z_2^{-1} z_1|^\alpha}<\infty
\right\}.
$$
For $k\in\mathbb N$ and $0<\alpha<1$, one defines
$\Gamma_{k+\alpha}(\Omega)$ as the space of all $f:\Omega\to\mathbb C$ such that all
complex derivatives of $f$ of order $k$ belong to $\Gamma_\alpha(\Omega)$.

A starting point for us will be that if $f\in \Gamma_\alpha(\Omega)$ for  $\alpha>0$ then
$u$ defined by \eqref{6} is twice differentiable in the complex directions and satisfies the
equation $\Box_{a,b} u= c_{a,b}f$. We refer to Folland and Stein \cite{FS},
Greiner and Stein \cite{GS}, and to Romero \cite{R} for three different approaches to
this property. Moreover, Folland and Stein have shown that if $f\in \Gamma_\alpha(\Omega, loc)$
and $\Box_{a,b} u= c_{a,b}f$, then $f\in \Gamma_{\alpha+2}(\Omega, loc)$.
These results extend those known for the Laplacian, in suitably redefined
anisotropic H\"older spaces.

We record relevant single and double layer potentials for the problem \eqref{EQ:BV}.
In \cite{J}, Jerison used the single layer potential defined by
$$
S_0 g(z)=\int_{\partial\Omega} g(\xi) \varepsilon(\xi,z) dS(\xi),
$$
which, however, is not integrable over characteristic points. On the contrary, the functional
$$
S g(z)=\int_{\partial\Omega} g(\xi) \varepsilon(\xi,z) \langle X_j, d\nu(\xi)\rangle,
$$
where $\langle X, d\nu\rangle$ is the canonical pairing between vector fields
and differential forms, is integrable over the whole boundary $\partial\Omega$. Moreover, it was shown in
\cite[Theorem 2.3]{R} that if the density of
$g(\xi)\langle X_j, d\nu\rangle$
in the operator $S$ is bounded then
$Sg\in\Gamma_\alpha(\mathbb{H}_{n-1})$ for all $\alpha<1$.
Parallel to $S$, it is natural to
use the operator
\begin{equation}\label{EQ:dp}
Wu(z)=\int_{\partial \Omega} u(\xi)\langle \nabla^{a,b}\varepsilon(\xi,z),d\nu(\xi)\rangle
\end{equation}
as a double layer potential.
Our main result for the sub-Laplacian is the following justification of formula
\eqref{EQ:bc0} in the introduction:

\begin{thm} \label{THM:main}
Let $\varepsilon(\xi,z)=\varepsilon(\xi^{-1}z)$ be the rescaled fundamental solution to
$\Box_{a,b}$, so that
\begin{equation}\label{EQ:def-eps}
\Box_{a,b}\varepsilon=c_{a,b}\delta\quad \textrm{ on } \mathbb{H}_{n-1}.
\end{equation}
For any $f\in \Gamma_\alpha(\Omega)$, the Newton potential \eqref{6} is the unique solution
in $C^{2}(\Omega)\cap C^1(\overline{\Omega})$
of the equation
\begin{equation}
\Box_{a,b} u= c_{a,b}f
\label{EQ:BV1}
\end{equation}
with the boundary condition
\begin{multline}\label{7}
(c_{a,b}-H.R(z))u(z)+\lim_{\delta\to 0}\int_{\partial \Omega \backslash\{|\xi^{-1} z|<\delta\}} u(\xi)\langle \nabla^{a,b}\varepsilon(\xi,z),d\nu(\xi)\rangle-
\\
\int_{\partial \Omega} \varepsilon(\xi,z) \langle \nabla^{b,a}u(\xi), d\nu(\xi)\rangle =0,
\qquad \textrm{ for } z\in\partial\Omega,
\end{multline}
where $H.R(z)$ is the so-called half residue given by the formula
\begin{equation}\label{EQ:HR}
H.R(z)=\lim_{\delta\rightarrow 0}\int_{\partial \Omega\backslash\{|\xi^{-1}z|<\delta\}}
\langle \nabla^{a,b}\varepsilon(\xi,z),d\nu(\xi)\rangle,
\end{equation}
with
$$\nabla^{a,b}g=\sum_{j=1}^{n-1} (aX_{j}gX_{\overline{j}}+b X_{\overline{j}}g X_{j}).$$
\end{thm}

The half residue $H.R. (z)$ in \eqref{EQ:HR} appears in the jump relations for the problem
\eqref{EQ:BV1} in the following way.
The double layer potential $Wu$ in \eqref{EQ:dp} has two limits
$$
W^+u(z)=\lim_{\stackrel{z_0\to z}{z_0\in \Omega}}\int_{\partial \Omega} u(\xi)\langle \nabla^{a,b}\varepsilon(\xi,z_0),d\nu(\xi)\rangle
$$
 and
$$
W^-u(z)=\lim_{\stackrel{z_0\to z}{z_0\not\in \Omega}}\int_{\partial \Omega} u(\xi)\langle \nabla^{a,b}\varepsilon(\xi,z_0),d\nu(\xi)\rangle,
$$
and the principal value
$$
W^0 u(z)= p.v. \;W u(z)=
\lim_{\delta\to 0}\int_{\partial \Omega \backslash\{|\xi^{-1} z|<\delta\}}
u(\xi)\langle \nabla^{a,b}\varepsilon(\xi,z),d\nu(\xi)\rangle.
$$
We note that this principal value enters as the second term in the integral boundary condition \eqref{7}.
It was proved in \cite[Theorem 2.4]{R} that for sufficiently regular $u$
(e.g. $u\in\Gamma_\alpha(\Omega)$) and $z\in \partial\Omega$
these limits exist and satisfy the jump relations
\begin{eqnarray}\nonumber
W^+ u(z)-W^-u(z)& = & c_{a,b} u(z), \\ \nonumber
W^0 u(z)-W^-u(z)& = & H.R.(z) u(z), \\
W^+u(z)-W^0 u(z)& = & (c_{a,b}-H.R.(z)) u(z), \label{EQ:jr1}
\end{eqnarray}
the last property \eqref{EQ:jr1} following from the first two by subtraction.

\begin{proof}[Proof of Theorem \ref{THM:main}]
Since the solid potential
\begin{equation}\label{EQ:solidp}
u(z)=\int_{\Omega}f(\xi)\varepsilon(\xi,z)d\nu(\xi)
\end{equation}
is a solution of \eqref{EQ:BV1}, from the aforementioned results of Folland and Stein it follows that
$u$ is locally in $\Gamma_{\alpha+2}(\Omega,loc)$ and that it is twice complex differentiable in
$\Omega$. In particular, it follows that $u\in C^{2}(\Omega)\cap C^1(\overline{\Omega})$.

The following representation formula can be derived from the generalised second Green's formula
(see Theorem 4.5 in \cite{R} and cf. \cite{M}),
for $u\in C^{2}(\Omega)\cap C^1(\overline{\Omega})$ we have
\begin{multline}\label{8}
c_{a,b}u(z)=c_{a,b}\int_{\Omega}f(\xi)\varepsilon(\xi,z)d\nu(\xi)
+\int_{\partial \Omega} u(\xi)\langle \nabla^{a,b}\varepsilon(\xi,z),d\nu(\xi)\rangle-
\\
\int_{\partial \Omega} \varepsilon(\xi,z) \langle\nabla^{b,a}u(\xi),d\nu(\xi)\rangle,
\quad \textrm{ for any } z\in\Omega.
\end{multline}
Since $u(z)$ given by \eqref{EQ:solidp}
is a solution of \eqref{EQ:BV1}, using it in  \eqref{8} we get
\begin{multline}\label{EQ:aux1}
\int_{\partial \Omega} u(\xi)\langle \nabla^{a,b}\varepsilon(\xi,z),d\nu(\xi)\rangle-\int_{\partial \Omega} \varepsilon(\xi,z) \langle\nabla^{b,a}u(\xi),d\nu(\xi)\rangle=0,\;\\
\textrm{ for any } z\in\Omega.
\end{multline}

It is easy to see that the fundamental solution, i.e. the
function $\varepsilon(z)$ in \eqref{3} is homogeneous of degree $-2n+2$, that is
$$
\varepsilon(\lambda z)=\lambda^{-2a-2b}\varepsilon(z)=\lambda^{-2n+2}\varepsilon(z)\quad
\textrm{ for any } \lambda>0,
$$
since $a+b=n-1$. It follows that $\varepsilon$ and its first order complex derivatives are locally integrable.
Since $\varepsilon(\xi,z)=\varepsilon(\xi^{-1}z)$, we obtain that as $z$ approaches the boundary,
we can pass to the limit in the second term in \eqref{EQ:aux1}.

By using this and the relation \eqref{EQ:jr1} as $z\in\Omega$ approaches the boundary
$\partial\Omega$ from inside, we find that
\begin{multline}\label{EQ:BC}
(c_{a,b}-H.R(z))u(z)+\lim_{\delta\to 0}\int_{\partial \Omega \backslash\{|\xi^{-1} z|<\delta\}} u(\xi)\langle \nabla^{a,b}\varepsilon(\xi,z),d\nu(\xi)\rangle-
\\
\int_{\partial \Omega} \varepsilon(\xi,z) \langle\nabla^{b,a}u(\xi),d\nu(\xi)\rangle=0,
\quad \textrm{ for any } z\in\partial\Omega.
\end{multline}

This shows that \eqref{6} is a solution of the boundary value problem \eqref{EQ:BV1} with the boundary condition
\eqref{7}.

\medskip
Now let us prove its uniqueness.
If the boundary value problem has two solutions $u$ and $u_{1}$
then the function $w=u-u_{1}\in C^{2}(\Omega)\cap C^1(\overline{\Omega})$ satisfies the homogeneous equation
\begin{equation}
\Box_{a,b}w=0\,\,\,\textrm{ in } \Omega, \label{9}
\end{equation}
and the boundary condition \eqref{7}, i.e.
\begin{equation}
(c_{a,b}-H.R(z))w(z)+\lim_{\delta\to 0}\int_{\partial \Omega \backslash\{|\xi^{-1} z|<\delta\}} w(\xi)\langle \nabla^{a,b}\varepsilon(\xi,z),d\nu(\xi)\rangle-\label{10}\end{equation}
$$
\int_{\partial \Omega} \varepsilon(\xi,z) \langle\nabla^{b,a}w(\xi),d\nu(\xi)\rangle=0,
$$
for any $z\in\partial\Omega.$

Since $f\equiv0$ in this case instead of \eqref{8} we have the following representation formula
\begin{equation}c_{a,b}w(z)=\int_{\partial \Omega} w(\xi)\langle \nabla^{a,b}\varepsilon(\xi,z),d\nu(\xi)\rangle-
\int_{\partial \Omega} \varepsilon(\xi,z) \langle\nabla^{b,a}w(\xi),d\nu(\xi)\rangle\label{11}\end{equation}
for any $z\in\Omega$.
As above, by using the properties of the double and single layer potentials as $z\rightarrow \partial\Omega$, we obtain
\begin{equation}
c_{a,b}w(z)=(c_{a,b}-H.R(z))w(z)+
\label{12}\end{equation}
$$\lim_{\delta\to 0}\int_{\partial \Omega \backslash\{|\xi^{-1} z|<\delta\}} w(\xi)\langle \nabla^{a,b}\varepsilon(\xi,z),d\nu(\xi)\rangle-
\int_{\partial \Omega} \varepsilon(\xi,z) \langle\nabla^{b,a}w,d\nu(\xi)\rangle
$$
for any $z\in\partial\Omega.$
Comparing this with \eqref{10} we arrive at
\begin{equation}
w(z)=0,  \,\, z\in\partial\Omega.\label{13}
\end{equation}

The homogeneous equation \eqref{9} with the Dirichlet boundary condition \eqref{13} has only trivial solution
$w\equiv 0$ in $\Omega$, see e.g. \cite[Theorem 4.3]{R}.
This shows that the boundary value problem \eqref{EQ:BV1} with the boundary condition
\eqref{7} has a unique solution in $C^{2}(\Omega)\cap C^1(\overline{\Omega})$.
This completes the proof of Theorem \ref{THM:main}.
\end{proof}

\section{Powers of the Kohn Laplacian}
\label{SEC:higher}

As before, let $\Omega \subset \mathbb{H}_{n-1}$ be
an open bounded domain with a smooth boundary $\partial\Omega \in C^{\infty}.$
For $m\in\mathbb N$, we denote $\Box_{a,b}^{m}:=\Box_{a,b}\Box_{a,b}^{m-1}$.
Then for $m=1,2,\ldots$, we consider
the equation
\begin{equation}
\Box_{a,b}^{m}u(z)=c_{a,b}f(z), \,\,z\in\Omega.
\label{17}
\end{equation}

Let $\varepsilon(\xi,z)=\varepsilon(\xi^{-1}z)$ be the
rescaled fundamental solution of the Kohn Laplacian as in \eqref{EQ:def-eps}.
Let us now define
\begin{equation}
u(z)=\int_{\Omega}f(\xi)\varepsilon_{m}(\xi,z)d\nu(\xi)
\label{18}
\end{equation}
in $\Omega\subset\mathbb{H}_{n-1}$, where
$\varepsilon_{m}(\xi,z)$ is a rescaled fundamental solution of \eqref{17} such that
$$
\Box_{a,b}^{m-1}\varepsilon_{m}=\varepsilon.
$$

We take, with a proper distributional interpretation, for $m=2,3,\ldots$,
\begin{equation}\varepsilon_{m}(\xi,z)=\int_{\Omega}\varepsilon_{m-1}(\xi,\zeta)\varepsilon(\zeta,z)d\nu(\zeta),\qquad
\xi,z\in \Omega, \label{19}\end{equation}
with
$$\varepsilon_{1}(\xi,z)=\varepsilon(\xi,z).$$

A simple calculation shows that the generalised Newton potential \eqref{18} is a solution of \eqref{17}
in $\Omega$. The aim of this section is to find a boundary condition on $\partial\Omega$ such that with this boundary
condition the equation \eqref{17} has a unique solution in $C^{2m}(\Omega)$, which coincides with \eqref{18}.

Although fundamental solutions for higher order hypoelliptic operators on the Heisenberg
group may not have unique fundamental solutions, see Geller \cite{geller_83}, in the case of
the iterated sub-Laplacian $\Box_{a,b}^{m}$ we still have the uniqueness for our
problem in the sense of the following theorem, and the uniqueness argument in its proof.

\begin{thm} \label{THM:main2}
For any $f\in \Gamma_{\alpha}(\Omega)$, the generalised Newton potential \eqref{18} is a unique solution of the equation \eqref{17} in $C^{2m}(\Omega)\cap C^{2m-1}(\overline{\Omega})$
with $m$ boundary conditions
\begin{multline}
(c_{a,b}-H.R(z))\Box_{a,b}^{i}u(z)+ \\
\sum_{j=0}^{m-i-1}\lim_{\delta\to 0}\int_{\partial \Omega \backslash\{|\xi^{-1} z|<\delta\}}
\Box_{a,b}^{j+i}u(\xi)\langle\nabla^{a,b}\Box_{a,b}^{m-1-j}\varepsilon_{m}(\xi,z),d\nu(\xi)\rangle
-\\
\sum_{j=0}^{m-i-1}\int_{\partial\Omega}\Box_{a,b}^{m-1-j}
\varepsilon_{m}(\xi,z)\langle\nabla^{b,a}\Box_{a,b}^{j+i}u(\xi)d\nu(\xi)\rangle=0,\quad
z\in \partial\Omega,
\label{20}
\end{multline}
for all $i=0,1,\ldots,m-1,$
where
$$\nabla^{a,b}g=\sum_{j=1}^{n-1} (aX_{j}gX_{\overline{j}}+b X_{\overline{j}}g X_{j})$$
and $H.R(z)$ is the half residue given by the formula \eqref{EQ:HR}.
%$$H.R(z)=lim_{\epsilon\rightarrow 0}\int_{\partial \Omega-\{|\xi^{-1}z|<\epsilon\}}\langle \nabla^{a,b}\varepsilon(\xi,z),d\nu(\xi)\rangle.$$
\end{thm}

\begin{proof}
By applying Green's second formula for each $z\in \Omega$,
as in \eqref{8}, we obtain
\begin{equation}c_{a,b}u(z)=c_{a,b}\int_{\Omega}f(\xi)\varepsilon_{m}(\xi,z)d\nu(\xi)=
\int_{\Omega}\Box_{a,b}^{m}u(\xi)\varepsilon_{m}(\xi,z)d\nu(\xi)=\end{equation}

$$\int_{\Omega}\Box_{a,b}^{m-1}u(\xi)\Box_{a,b}\varepsilon_{m}(\xi,z)d\nu(\xi)-
\int_{\partial\Omega}\Box_{a,b}^{m-1}u(\xi)\langle \nabla^{a,b}\varepsilon_{m}(\xi,z),d\nu(\xi)\rangle+$$

$$\int_{\partial\Omega}\varepsilon_{m}(\xi,z)\langle \nabla^{b,a}\Box_{a,b}^{m-1}u(\xi),d\nu(\xi)\rangle=
\int_{\Omega}\Box_{a,b}^{m-2}u(\xi)\Box_{a,b}^{2}\varepsilon_{m}(\xi,z)d\nu(\xi)-$$

$$\int_{\partial\Omega}\Box_{a,b}^{m-2}u(\xi)\langle\nabla^{a,b}\Box_{a,b}\varepsilon_{m}(\xi,z),d\nu(\xi)\rangle
+$$
$$\int_{\partial\Omega}\Box_{a,b}\varepsilon_{m}(\xi,z)\langle\nabla^{b,a}\Box_{a,b}^{m-2}u(\xi),d\nu(\xi)\rangle-$$

$$\int_{\partial\Omega}\Box_{a,b}^{m-1}u(\xi)\langle \nabla^{a,b}\varepsilon_{m}(\xi,z),d\nu(\xi)\rangle+$$

$$\int_{\partial\Omega}\varepsilon_{m}(\xi,z)\langle \nabla^{b,a}\Box_{a,b}^{m-1}u(\xi),d\nu(\xi)\rangle=...=$$

$$c_{a,b}u(z)-\sum_{j=0}^{m-1}\int_{\partial\Omega}\Box_{a,b}^{j}u(\xi)\langle\nabla^{a,b}
\Box_{a,b}^{m-1-j}\varepsilon_{m}(\xi,z),d\nu(\xi)\rangle+$$
$$
\sum_{j=0}^{m-1}\int_{\partial\Omega}\Box_{a,b}^{m-1-j}\varepsilon_{m}(\xi,z)\langle\nabla^{b,a}\Box_{a,b}^{j}u(\xi),d\nu(\xi)\rangle,\quad z\in \Omega.
$$
This implies the identity
\begin{multline}\label{22}
\sum_{j=0}^{m-1}\int_{\partial\Omega}\Box_{a,b}^{j}u(\xi)\langle\nabla^{a,b}
\Box_{a,b}^{m-1-j}\varepsilon_{m}(\xi,z),d\nu(\xi)\rangle- \\
\sum_{j=0}^{m-1}\int_{\partial\Omega}\Box_{a,b}^{m-1-j}\varepsilon_{m}(\xi,z)\langle\nabla^{b,a}\Box_{a,b}^{j}u(\xi),d\nu(\xi)\rangle=0,\quad z\in \Omega.
\end{multline}

By using the properties of the double and single layer potentials as $z$
approaches the boundary $\partial\Omega$ from the interior, from \eqref{22} we obtain
\begin{multline*}
(c_{a,b}-H.R(z))u(z)+\sum_{j=0}^{m-1}\lim_{\delta\to 0}\int_{\partial \Omega \backslash\{|\xi^{-1} z|<\delta\}}\Box_{a,b}^{j}u(\xi)\langle\nabla^{a,b}
\Box_{a,b}^{m-1-j}\varepsilon_{m}(\xi,z),d\nu(\xi)\rangle -\\
\sum_{j=0}^{m-1}\int_{\partial\Omega}\Box_{a,b}^{m-1-j}\varepsilon_{m}(\xi,z)\langle\nabla^{b,a}\Box_{a,b}^{j}u(\xi),d\nu(\xi)\rangle=0,\quad z\in \partial\Omega.
\end{multline*}
Thus, this relation is one of the boundary conditions of \eqref{18}.
Let us derive the remaining boundary conditions. To this end, we write
\begin{equation}
\Box_{a,b}^{m-i}\Box_{a,b}^{i}u=c_{a,b}f,\quad i=0,1,\ldots,m-1,\quad m=1,2,\ldots,
\label{24}
\end{equation}
and carry out similar considerations just as above. This yields

$$c_{a,b}\Box_{a,b}^{i}u(z)=c_{a,b}\int_{\Omega}f(\xi)\Box_{a,b}^{i}\varepsilon_{m}(\xi,z)d\nu(\xi)=$$
$$\int_{\Omega}\Box_{a,b}^{m-i}\Box_{a,b}^{i}u(\xi)\Box_{a,b}^{i}\varepsilon_{m}(\xi,z)d\nu(\xi)=$$
$$\int_{\Omega}\Box_{a,b}^{m-i-1}\Box_{a,b}^{i}u(\xi)\Box_{a,b}\Box_{a,b}^{i}\varepsilon_{m}(\xi,z)d\nu(\xi)-$$
$$
\int_{\partial\Omega}\Box_{a,b}^{m-i-1}\Box_{a,b}^{i}u(\xi)\langle\nabla^{a,b}\Box_{a,b}^{i}\varepsilon_{m}(\xi,z),d\nu(\xi)\rangle+$$
$$\int_{\partial\Omega}\Box_{a,b}^{i}\varepsilon_{m}(\xi,z)\langle \nabla^{b,a}\Box_{a,b}^{m-i-1}\Box_{a,b}^{i}u(\xi),d\nu(\xi)\rangle=
$$
$$
\int_{\Omega}\Box_{a,b}^{m-i-2}\Box_{a,b}^{i}u(\xi)\Box_{a,b}^{2}\Box_{a,b}^{i}\varepsilon_{m}(\xi,z)d\nu(\xi)-$$
$$\int_{\partial\Omega}\Box_{a,b}^{m-i-2}\Box_{a,b}^{i}u(\xi)\langle\nabla^{a,b}\Box_{a,b}\Box_{a,b}^{i}\varepsilon_{m}(\xi,z),d\nu(\xi)\rangle+$$
$$\int_{\partial\Omega}\Box_{a,b}\Box_{a,b}^{i}\varepsilon_{m}(\xi,z)\langle\nabla^{b,a}\Box_{a,b}^{m-i-2}\Box_{a,b}^{i}u(\xi),d\nu(\xi)\rangle-$$
$$\int_{\partial\Omega}\Box_{a,b}^{m-i-1}\Box_{a,b}^{i}u(\xi)\langle\nabla^{a,b}\Box_{a,b}^{i}\varepsilon_{m}(\xi,z),d\nu(\xi)\rangle+$$
$$\int_{\partial\Omega}\Box_{a,b}^{i}\varepsilon_{m}(\xi,z)\langle\nabla^{b,a}\Box_{a,b}^{m-i-1}\Box_{a,b}^{i}u(\xi),d\nu(\xi)\rangle=$$

$$...=\int_{\Omega}\Box_{a,b}^{i}u(\xi)\Box_{a,b}^{m-i}\Box_{a,b}^{i}\varepsilon_{m}(\xi,z)d\nu(\xi)-$$
$$\sum_{j=0}^{m-i-1}\int_{\partial\Omega}
\Box_{a,b}^{j}\Box_{a,b}^{i}u(\xi)\langle\nabla^{a,b}\Box_{a,b}^{m-i-1-j}\Box_{a,b}^{i}\varepsilon_{m}(\xi,z),d\nu(\xi)\rangle+$$
$$\sum_{j=0}^{m-i-1}\int_{\partial\Omega}\Box_{a,b}^{m-i-1-j}\Box_{a,b}^{i}\varepsilon_{m}(\xi,z)
\langle\nabla^{b,a}\Box_{a,b}^{j}\Box_{a,b}^{i}u(\xi),d\nu(\xi)\rangle=$$
$$c_{a,b}\Box_{a,b}^{i}u(z)-\sum_{j=0}^{m-i-1}\int_{\partial\Omega}
\Box_{a,b}^{j+i}u(\xi)\langle\nabla^{a,b}\Box_{a,b}^{m-1-j}\varepsilon_{m}(\xi,z),d\nu(\xi)\rangle+$$
$$\sum_{j=0}^{m-i-1}\int_{\partial\Omega}\Box_{a,b}^{m-1-j}\varepsilon_{m}(\xi,z)
\langle\nabla^{b,a}\Box_{a,b}^{j+i}u(\xi),d\nu(\xi)\rangle,\quad z\in\Omega,
$$
where, as usual, $\varepsilon_{m}(\xi,z)=\varepsilon_{m}(\xi^{-1}z)$, and
$\Box_{a,b}^{i}\varepsilon_{m}$ is a
rescaled fundamental solution of the equation \eqref{24}, i.e.,
$$
\Box_{a,b}^{m-i}\Box_{a,b}^{i}\varepsilon_{m}=c_{a,b} \delta,\qquad i=0,1,\ldots,m-1.
$$

From the previous relations, we obtain the identities
\begin{multline*}
\sum_{j=0}^{m-i-1}\int_{\partial\Omega}
\Box_{a,b}^{j+i}u(\xi)\langle\nabla^{a,b}\Box_{a,b}^{m-1-j}\varepsilon_{m}(\xi,z),d\nu(\xi)\rangle
-\\
\sum_{j=0}^{m-i-1}\int_{\partial\Omega}\Box_{a,b}^{m-1-j}\varepsilon_{m}(\xi,z)
\langle\nabla^{b,a}\Box_{a,b}^{j+i}u(\xi),d\nu(\xi)\rangle=0
\end{multline*}
for any $z\in\Omega$ and $i=0,1,\ldots,m-1.$
By using the properties of the double and single layer potentials as
$z$ approaches the boundary $\partial\Omega$ from the interior of $\Omega$, we find that
\begin{multline*}\label{25}
(c_{a,b}-H.R(z))\Box_{a,b}^{i}u(z)+\sum_{j=0}^{m-i-1}\lim_{\delta\to 0}\int_{\partial \Omega \backslash\{|\xi^{-1} z|<\delta\}}
\Box_{a,b}^{j+i}u(\xi)\langle\nabla^{a,b}\Box_{a,b}^{m-1-j}\varepsilon_{m}(\xi,z),d\nu(\xi)\rangle
-\\
\sum_{j=0}^{m-i-1}\int_{\partial\Omega}\Box_{a,b}^{m-1-j}\varepsilon_{m}(\xi,z)
\langle\nabla^{b,a}\Box_{a,b}^{j+i}u(\xi),d\nu(\xi)\rangle=0,\quad
z\in\partial\Omega,
\end{multline*}
are all boundary conditions of \eqref{18} for each $i=0,1,\ldots,m-1$.

\medskip
Conversely, let us show that if a function $w\in C^{2m}(\Omega)\cap C^{2m-1}(\overline{\Omega})$ satisfies the equation $\Box_{a,b}^{m}w=f$  and
the boundary conditions \eqref{20}, then it coincides with the solution \eqref{18}.
Indeed, otherwise the function
$$v=u-w\in C^{2m}(\Omega)\cap C^{2m-1}(\overline{\Omega}),$$
where $u$ is the generalised Newton potential \eqref{18}, satisfies the homogeneous equation
\begin{equation}
\Box_{a,b}^{m}v=0\label{26}
\end{equation}
and the boundary conditions \eqref{20}, i.e.
\begin{multline*}
I_{i}(v)(z):=
(c_{a,b}-H.R(z))\Box_{a,b}^{i}v(z)  \\ +
\sum_{j=0}^{m-i-1}\lim_{\delta\to 0}\int_{\partial \Omega \backslash\{|\xi^{-1} z|<\delta\}}
\Box_{a,b}^{j+i}v(\xi)\langle\nabla^{a,b}\Box_{a,b}^{m-1-j}\varepsilon_{m}(\xi,z),d\nu(\xi)\rangle
\\ -
\sum_{j=0}^{m-i-1}\int_{\partial\Omega}\Box_{a,b}^{m-1-j}\varepsilon_{m}(\xi,z)
\langle\nabla^{b,a}\Box_{a,b}^{j+i}v(\xi),d\nu(\xi)\rangle=0,\quad i=0,1,\ldots,m-1,
\end{multline*}
for $z\in\partial\Omega.$
By applying the Green formula to the function $v\in C^{2m}(\Omega)\cap C^{2m-1}(\overline{\Omega})$ and by following the lines of the above
argument, we obtain
$$
0=\int_{\Omega}\Box_{a,b}^{m}v(z)\Box_{a,b}^{i}\varepsilon_{m}(\xi,z)d\nu(\xi)=
\int_{\Omega}\Box_{a,b}^{m-i}\Box_{a,b}^{i}v(z)\Box_{a,b}^{i}\varepsilon_{m}(\xi,z)d\nu(\xi)=$$
$$
\int_{\Omega}\Box_{a,b}^{m-1}v(z)\Box_{a,b}\Box_{a,b}^{i}\varepsilon_{m}(\xi,z)d\nu(\xi)-$$
$$\int_{\partial\Omega}
\Box_{a,b}^{m-1}v(z)\langle\nabla^{a,b}\Box_{a,b}^{i}\varepsilon_{m}(\xi,z),d\nu(\xi)\rangle+$$
$$\int_{\partial\Omega}\Box_{a,b}^{i}\varepsilon_{m}(\xi,z)
\langle\nabla^{a,b}\Box_{a,b}^{m-1}v(z),d\nu(\xi)\rangle=...=$$
$$c_{a,b}\Box_{a,b}^{i}v(z)-\sum_{j=0}^{m-i-1}\int_{\partial\Omega}
\Box_{a,b}^{j+i}v(\xi)\langle\nabla^{a,b}\Box_{a,b}^{m-1-j}\varepsilon_{m}(\xi,z),d\nu(\xi)\rangle+$$
$$\sum_{j=0}^{m-i-1}\int_{\partial\Omega}\Box_{a,b}^{m-1-j}\varepsilon_{m}(\xi,z)
\langle\nabla^{b,a}\Box_{a,b}^{j+i}v(\xi),d\nu(\xi)\rangle, \quad i=0,1,\ldots,m-1.$$

By passing to the limit as $z\rightarrow \partial\Omega$, we obtain the relations
\begin{equation}
\Box_{a,b}^{i}v(z)\mid_{z\in\partial\Omega}=I_{i}(v)(z)\mid_{z\in\partial\Omega}=0,\quad
i=0,1,\ldots,m-1.
\end{equation}

Assuming for the moment the uniqueness of the solution of the boundary value problem
\begin{equation}\label{EQ:DPm}
\Box_{a,b}^{m}v=0,
\end{equation}
$$
\Box_{a,b}^{i}v\mid_{\partial\Omega}=0, \quad i=0,1,\ldots,m-1,
$$
we get that $v=u-w\equiv0$, for all $z\in\Omega$, i.e. $w$ coincides with $u$ in $\Omega$.
Thus \eqref{18} is the unique solution of the boundary value problem
\eqref{17}, \eqref{20} in $\Omega$.

It remains to argue that the boundary value problem \eqref{EQ:DPm} has a unique solution
in $C^{2m}(\Omega)\cap C^{2m-1}(\overline{\Omega})$. Denoting
$\tilde v:= \Box_{a,b}^{m-1}v$,
this follows by induction
from the uniqueness in $C^{2}(\Omega)\cap C^{1}(\overline{\Omega})$ of the problem
$$
\Box_{a,b} \tilde v=0,\quad \tilde v\mid_{\partial\Omega}=0.
$$

The proof of Theorem \ref{THM:main2} is complete.
\end{proof}

\begin{rem}
It follows from Theorem \ref{THM:main2}  that the kernel \eqref{19}, which is a
rescaled fundamental solution of the equation \eqref{17}, is the Green function of the
boundary value problem \eqref{17}, \eqref{20} in $\Omega$. Therefore, the
boundary value problem
\eqref{17}, \eqref{20} can serve as an example of an explicitly solvable
boundary value problem in any
domain $\Omega$ (with smooth boundary) on the Heisenberg group.
\end{rem}

     \end{document}